\documentclass[10pt]{article}
\usepackage{amsthm,,float}
\usepackage{epstopdf}
\usepackage{amsmath}
\usepackage{graphicx}
\usepackage{amsfonts}
\usepackage{amssymb,enumerate}
\usepackage{amsthm, amstext}
\usepackage{amsmath}
\usepackage{amsmath, amsthm, amssymb}

\usepackage{enumitem}
\usepackage{graphicx}
\usepackage{amsfonts}
\usepackage{latexsym, xypic}
\theoremstyle{plain}
\newtheorem{theorem}{Theorem}[section]
\newtheorem{lemma}[theorem]{Lemma}

\theoremstyle{definition}
\newtheorem{definition}[theorem]{Definition}
\newtheorem{corollary}[theorem]{Corollary}
\newtheorem{example}{\sc Example}
\theoremstyle{remark}

\usepackage[T1]{fontenc}
\usepackage{times}%

\begin{document}
	
\title{\bf {Pronormal $L$-subgroups of an $L$-group}}
\author{\textbf{Iffat Jahan$^1$ and Ananya Manas $^2$} \\\\ 
	$^{1}$ Department of Mathematics, Ramjas College\\
	University of Delhi, Delhi, India \\
	ij.umar@yahoo.com \\\\
	$^{2}$Department of Mathematics, \\
	University of Delhi, Delhi, India \\
	ananyamanas@gmail.com 
		\\    }
	\date{}
	\maketitle

\begin{abstract}
	 \noindent In this paper, the notion of pronormal $L$-subgroups of an $L$-group has been introduced by using the concept of conjugate $L$-subgroup developed in \cite{jahan_conj}. The notion of pronormal $L$-subgroups has been investigated in context of normality and subnormality of $L$-subgroups and several related properties have been established. Moreover, the relation of pronormality with normalizers and maximal $L$-subgroups has been explored. \\ \\
	 {\bf Keywords:} $L$-algebra; $L$-subgroup; Generated $L$-subgroup; Pronormal $L$-subgroup; Normal $L$-subgroup; Conjugate of an $L$-subgroup; Maximal $L$-subgroup.
\end{abstract}
		
\section{Introduction}
	\noindent  The notion of a fuzzy set  was founded by Zadeh \cite{zadeh_fuzzy} in 1965 and Rosenfeld \cite{rosenfeld_fuzzy} applied this notion to group theory in the year 1971 which led to the evolution of fuzzy group theory. In 1981, Liu \cite{liu_op} introduced the notion of lattice valued fuzzy subgroups ($L$-subgroups). This pioneered the studies of $L$(lattice valued fuzzy)-algebraic 
		substructures. Recently, in a series of papers \cite{ajmal_char,ajmal_nc,ajmal_nil,ajmal_nor,ajmal_sol,jahan_max} various concepts of classical group theory such as characteristic subgroups, normalizer of a subgroup, nilpotent subgroups, solvable subgroups, normal closure of a subgroup, maximal subgroup etc. have been studied within the framework of $L$-setting and are shown to be compatible. Thus a coherent and systematic theory is coming into existence with this development.

		In classical group theory, the notion of pronormal subgroups is closely related to the notions of
		normal and subnormal subgroups. In fact, pronormality together with subnormality is equivalent to normality. The pronormal fuzzy subgroups were introduced by Abou Zaid [9] using the concept of level subsets. However, his definition and the corresponding study
		of pronormal fuzzy subgroups fail to provide any information about the deeper structure of
		fuzzy pronormal subgroups. Hence a fresh approach to pronormal $L$(fuzzy)-subgroups is needed.
		In \cite{jahan_conj}, the authors have introduced the notion of conjugate of an $L$-subgroup of an $L$-group by an $L$-point. This notion has been utilized to develop pronormal $L$-subgroups of an $L$-group in this paper.

		We begin our work in section 3 by defining the pronormal $L$-subgroups of an $L$-group using the notion of conjugate $L$-subgroups developed in \cite{jahan_conj}. An example has been provided to show the existence of pronormal $L$-subgroups. Then, a connection with the classical property of pronormality of level subsets with that of pronormal $L$-subgroups has been established.  Then, it has been shown that the image of a pronormal $L$-subgroup under group homomorphism is a pronormal $L$-subgroup, provided that the parent group $\mu$ possesses the sup-property.
		
		In section 4, we have explored the important relationships between the notions of normality, subnormality, normalizer and maximality with the concept of pronormality defined in $L$-group theory. Firstly, it has been shown that every normal $L$-subgroup of an $L$-group $\mu$ is a pronormal $L$-subgroup of $\mu$. Then, the pronormality of the normalizer of a pronormal $L$-subgroup has been discussed. It has been exhibited that a maximal $L$-subgroup of an $L$-group is a pronormal $L$-subgroup. The notion of subnormal $L$-subgroup of an $L$-group was introduced in \cite{ajmal_nc} and studied in detail in \cite{ajmal_subnormal}. This notion is 
		used to establish that an $L$-subgroup of an $L$-group $\mu$ that is both a subnormal and a pronormal $L$-subgroup of $\mu$ is a normal $L$-subgroup of $\mu$. This result is then applied to prove that every pronormal $L$-subgroup of a nilpotent $L$-group $\mu$ having the same tip and tail as $\mu$ is normal in $\mu$.
		    
\section{Preliminaries}

Throughout this paper, $L = \langle L, \leq, \vee, \wedge  \rangle$ denotes a complete and completely distributive lattice where '$\leq$' denotes the partial ordering on $L$ and '$\vee$'and '$\wedge$' denote, respectively, the join (supremum) and meet (infimum) of the elements of $L$. Moreover, the maximal and minimal elements of $L$ will be denoted by $1$ and $0$, respectively. The concept of completely distributive lattice can be found in any standard text on the subject \cite{gratzer_lattices}. 

The notion of a fuzzy subset of a set was introduced by Zadeh \cite{zadeh_fuzzy} in 1965. In 1967, Goguen \cite{goguen_sets} extended this concept to $L$-fuzzy sets. In this section, we recall the basic definitions and results associated with $L$-subsets that shall be used throughout this work. These definitions can be found in chapter 1 of \cite{mordeson_comm}.

Let $X$ be a non-empty set. An $L$-subset of $X$ is a function from $X$  into $L$. The set of  $L$-subsets of $X$ is called  the $L$-power set of $X$ and is denoted by $L^X$.  For  $\mu \in L^X, $  the set $ \lbrace\mu(x) \mid x \in X \rbrace$  is called the image of $\mu$  and is denoted by  Im $\mu $. The tip and tail of $ \mu $  are defined as $\bigvee \limits_{x \in X}\mu(x)$ and $\bigwedge \limits_{x \in X}\mu(x)$, respectively. An $L$-subset $\mu$ of $X$ is said to be contained in an $L$-subset $\eta$  of $X$ if  $\mu(x)\leq \eta (x)$ for all $x \in X$. This is denoted by $\mu \subseteq \eta $.  For a family $\lbrace\mu_{i} \mid i \in I \rbrace$  of $L$-subsets in  $X$, where $I$  is a non-empty index set, the union $\bigcup\limits_{i \in I} \mu_{i} $    and the intersection  $\bigcap\limits_{i \in I} \mu_{i} $ of  $\lbrace\mu_{i} \mid i \in I \rbrace$ are, respectively, defined by
\begin{center}
	$\bigcup\limits_{i \in I} \mu_{i}(x)= \bigvee\limits_{i \in I} \mu(x) $ \quad and \quad $\bigcap\limits_{i \in I} \mu_{i} (x)= \bigwedge\limits_{i \in I} \mu(x) $
\end{center}
for each  $x \in X $. If  $\mu \in L^X $  and  $a \in L $,  then the level  subset $\mu_{a}$ of $\mu$  is defined as
\begin{center}
	$\mu_{a}= \lbrace x \in X \mid \mu (x) \geq a\rbrace$. 
\end{center}
\noindent For $\mu, \nu \in L^{X} $, it can be verified easily that if $\mu\subseteq \nu$, then $\mu_{a} \subseteq \nu_{a} $ for each $a\in L $.

For $a\in L$ and $x \in X$, we define $a_{x} \in L^{X} $ as follows: for all $y \in X$,
\[
a_{x} ( y ) =
\begin{cases}
	a &\text{if} \ y = x,\\
	0 &\text{if} \ y\ne x.
\end{cases}
\]
$a_{x} $ is referred to as an $L$-point or $L$-singleton. We say that $a_{x} $ is an $L$-point of $\mu$ if and only if
$\mu( x )\ge a$ and we write $a_{x} \in \mu$. 

Let $S$ be a groupoid. The set product $\mu \circ \eta$   of $\mu, \eta \in L^S$ is an $L$-subset of $S$ defined by
\begin{center}
	$\mu \circ \eta (x) = \bigvee \limits_{x=yz}\lbrace\mu (y) \wedge \eta (z) \rbrace.$
\end{center}

\noindent Note that if $x$ cannot be factored as  $x=yz$  in $S$, then  $\mu \circ \eta (x)$, being  the least upper bound of the empty set, is zero. It can be verified that the set product is associative in  $L^S$  if $S$ is a semigroup.

Let $f$ be a mapping from a set $X$ to a set $Y$. If $\mu \in L ^{X}$ and $\nu \in L^{Y}$, then the image $f(\mu )$
of $\mu $ under $f$ and the preimage $f^{-1} (\nu )$ of $\nu $ under $f$ are $L$-subsets of $Y$ and $X$ respectively, defined by
\[ f(\mu )(y)=\bigvee\limits_{x\in f^{-1} (y)} \{\mu (x)\} \]
and
\[ f^{-1} (\nu )(x)=\nu (f(x)). \]
Again,  if $f^{-1} (y)=\emptyset $,
then $f(\mu )(y)$ being the least upper bound of the empty set, is zero.

Throughout this paper, $G$ denotes an ordinary group with the identity element `$e$' and $I$ denotes a non-empty indexing set. Also, $1_A$ denotes the characteristic function of a non-empty set $A$.

In 1971, Rosenfeld \cite{rosenfeld_fuzzy} applied the notion of fuzzy sets to groups to introduce the fuzzy subgroup of a group. Liu \cite{liu_op}, in 1981, extended the notion of fuzzy subgroups to $L$-fuzzy subgroups ($L$-subgroups), which we define below.   

\begin{definition}(\cite{rosenfeld_fuzzy})
	Let $\mu \in L ^G $. Then, $\mu $ is called an $L$-subgroup of $G$ if for each $x, y\in G$,
	\begin{enumerate}
		\item[({i})] $\mu (xy)\ge \mu (x)\wedge \mu (y)$,
		\item[({ii})] $\mu (x^{-1} )=\mu (x)$.
	\end{enumerate}
	The set of $L$-subgroups of $G$ is denoted by $L(G)$. Clearly, the tip of an $L$-subgroup
	is attained at the identity element of $G$.
\end{definition}

\begin{theorem}(\cite{mordeson_comm}, Lemma 1.2.5)
	\label{lev_gp}
	Let $\mu \in L ^G $. Then, $\mu $ is an $L$-subgroup of $G$ if and only if each non-empty level subset $\mu_{a} $ is a subgroup of $G$.
\end{theorem}

It is well known in literature that the intersection of an arbitrary family of $L$-subgroups of a group is an $L$-subgroup of the given group.

\begin{definition}(\cite{rosenfeld_fuzzy})
	Let $\mu \in L ^G $. Then, the $L$-subgroup of $G$ generated by $\mu $ is defined as the smallest $L$-subgroup of $G$
	which contains $\mu $. It is denoted by $\langle \mu \rangle $, that is,
	\[
	\langle \mu \rangle =\cap\{\mu _{{i}} \in L(G) \mid \mu \subseteq \mu _{i}\}.
	\]
\end{definition}

Let $\eta, \mu\in L^{{G}}$ such that $\eta\subseteq\mu$. Then, $\eta$ is said to be an $L$-subset of $\mu$. The set of all $L$-subsets of $\mu$ is denoted by $L^{\mu}.$
Moreover, if $\eta,\mu\in L(G)$ such that  $\eta\subseteq \mu$, then $\eta$ is said to be an $L$-subgroup of $\mu$. The set of all $L$-subgroups of $\mu$ is denoted by $L(\mu)$.

From now onwards, $\mu$ denotes an $L$-subgroup of $G$ which shall be considered as the parent $L$-group. In fact, $\mu$ is an $L$-subgroup of $G$ if and only if $\mu$ is an $L$-subgroup of $1_G$.

\begin{definition}(\cite{ajmal_sol}) 
	Let $\eta\in L(\mu)$ such that $\eta$ is non-constant and $\eta\ne\mu$. Then, $\eta$ is said to be a proper $L$-subgroup of $\mu$.
\end{definition}

\noindent Clearly, $\eta$ is a proper $L$-subgroup of $\mu$ if and only if $\eta$ has distinct tip and tail and $\eta\ne\mu$.

\begin{definition}(\cite{ajmal_nil})
	Let $\eta \in L(\mu)$. Let $a_0$ and $t_0$ denote the tip and tail of $\eta$, respectively. We define the trivial $L$-subgroup of $\eta$ as follows:
	\[ \eta_{t_0}^{a_0}(x) = \begin{cases}
		a_0 & \text{if } x=e,\\
		t_0 & \text{if } x \neq e.
	\end{cases} \]
\end{definition}

\begin{theorem}(\cite{ajmal_nil}, Theorem 2.1)
	\label{lev_sgp}
	Let $\eta \in L^\mu$. Then,  $\eta\in L(\mu)$ if and only if each non-empty level subset $\eta_a$  is a subgroup of $\mu_a$.
	
\end{theorem}

The normal fuzzy subgroup of a fuzzy group was introduced by Wu \cite{wu_normal} in 1981. We note that for the development of this concept, Wu \cite{wu_normal} preferred the $L$-setting. Below, we recall the notion of a normal $L$-subgroup of an $L$-group:

\begin{definition}(\cite{wu_normal})
	Let $\eta \in L(\mu)$. Then, we say that  $\eta$  is a normal $L$-subgroup of $\mu$   if  
	\begin{center}
		$\eta(yxy^{-1}) \geq \eta(x)\wedge \mu(y)$ for  all  $x,y \in G.$
	\end{center}
\end{definition}

\noindent The set of normal $L$-subgroups of $\mu$  is denoted by $NL(\mu)$. If $\eta \in NL(\mu)$, then we write\vspace{.2cm} $ \eta \triangleleft \mu$. 

Here, we mention that the arbitrary intersection of a family of normal $L$-subgroups of an $L$-group $\mu$ is again a normal $L$-subgroup of $\mu$. 

\begin{theorem}(\cite{ajmal_char})
	\label{lev_norsgp}
	Let $\eta \in L(\mu)$. Then, $\eta\in NL(\mu)$ if and only if each non-empty level subset $\eta_a$  is a normal subgroup of $\mu_a$.
\end{theorem}

\begin{definition}(\cite{rosenfeld_fuzzy})
	Let $\mu \in L^X$. Then, $\mu$ is said to possess sup-propery if for each $A \subseteq X$, there exists $a_0 \in A$ such that $ \mathop \vee \limits_{a \in A}  {\mu(a) } = \mu(a_0)$. 
\end{definition}

\noindent Lastly, recall the following form \cite{ajmal_gen}:

\begin{theorem}(\cite{ajmal_gen}, Theorem 3.1)
	\label{gen_form}
	Let $\eta\in L^{^{\mu}}.$ Let $a_{0}=\mathop {\vee}\limits_{x\in G}{\left\{\eta\left(x\right)\right\}}$ and define an $L$-subset $\hat{\eta}$ of $G$ by
	\begin{center}
		$\hat{\eta}\left(x\right)=\mathop{\vee}\limits_{a \leq a_{0}}{\left\{a \mid x\in\left\langle \eta_{a}\right\rangle\right\}}$.
	\end{center}
	
	\noindent Then, $\hat{\eta}\in L(\mu)$ and  $\hat{\eta} =\left\langle \eta \right\rangle$.
\end{theorem}

\begin{theorem}(\cite{ajmal_gen}, Theorem 3.7)
	\label{gen_sup}
	Let $\eta \in L^{\mu}$ and possesses the sup-property. If $a_0 = \mathop{\vee}\limits_{x \in G}\{\eta(x)\}$, then for all $b \leq a_0$, $\langle \eta_b \rangle = \langle \eta \rangle_b$.
\end{theorem}

\section{Pronormal $L$-subgroups}

The notion of pronormal subgroups in classical group theory utilizes the concept of conjugate subgroups. In fuzzy group theory, the notion of conjugate fuzzy subgroup was introduced by Mukherjee and Bhattacharya \cite{mukherjee_some}. The conjugate developed therein is by a crisp point of the parent group $G$ rather than a fuzzy point and thus could not be applied in the development of pronormal fuzzy subgroups. Indeed, the pronormality of fuzzy subgroups, introduced by Abou-Zaid \cite{abouzaid_Pronormal}, was developed through level subsets, which does not reveal any information regarding their structure.

In \cite{jahan_conj}, the authors have introduced the conjugate of an $L$-subgroup of an $L$-group by an $L$-point. This definition has been shown to be highly compatible with other notions in $L$-group theory such as normal $L$-subgroups of an $L$-group, normalizer of an $L$-subgroup of an $L$-group \cite{ajmal_nor} and maximal $L$-subgroup of an $L$-group \cite{jahan_max}. Moreover, this definition removes the shortcomings of the conjugate introduced in \cite{mukherjee_some} and can easily be utilized to develop pronormal $L$-subgroups of an $L$-group.  

\begin{definition}(\cite{jahan_conj})
	Let $\eta$ be an $L$-subgroup of $\mu$ and $a_z$ be an $L$-point of $\mu$. The conjugate $\eta^{a_z}$ of $\eta$ with respect to $a_z$ is the $L$-subset of $G$ defined by
	\[ \eta^{a_z}(x) = a \wedge \eta(zxz^{-1}) \quad \text{ for all } x \in G. \]	
\end{definition}

\noindent We remark here that for an $L$-subgroup $\eta$ of $\mu$ and an $L$-point $a_z$ of $\mu$, the conjugate $\eta^{a_z}$ forms an $L$-subgroup of $\mu$. Moreover, $\text{tip}(\eta^{a_z}$) = $a \wedge \text{tip}(\eta)$, since
\[ \eta^{a_z}(e) = a \wedge \eta(zez^{-1}) = a \wedge \eta(e). \]

\noindent We also recall the level subset characterization for conjugate $L$-subgroups from \cite{jahan_conj}. Here, we note that for a subgroup $H$ of $G$ and for $x \in G$, $H^x$ denotes the conjugate of $H$ with respect to $x$.

\begin{theorem}(\cite{jahan_conj})\label{lvl_conj}
	Let $\eta, \nu \in L(\mu)$ and $a \in L$ such that $\text{tip}(\nu) = a \wedge \text{tip}(\eta)$. Then, $\nu=\eta^{a_z}$ for $a_z \in \mu$ if and only if $\nu_t = {\eta_t}^{z^{-1}}$ for all $t \leq \text{tip}(\nu)$.
\end{theorem}

\noindent We are now ready to define the notion of pronormal $L$-subgroups.

\begin{definition}
	Let $\mu \in L(G)$. An $L$-subgroup $\eta$ of $\mu$ is said to be a pronormal $L$-subgroup of $\mu$ if for every $L$-point $a_x \in \mu$, there exists an $L$-point $b_y \in \langle \eta, \eta^{a_x} \rangle$ such that $\eta^{b_y} = \eta^{a_x}$.
\end{definition}

\noindent Here, we note that for $L$-subgroups $\eta$ and $\nu$ of $\mu$, $\langle \eta, \nu \rangle$ denotes the $L$-subgroup of $\mu$ generated by $\eta \cup \nu$. 

Our definition of pronormal $L$-subgroups is motivated by the following:

\begin{theorem}
	Let $H$ and $K$ be subgroups of $G$ such that $H \subseteq K$. Then, $H$ is a pronormal subgroup of $K$ if and only if $1_H$ is a pronormal $L$-subgroup of $1_K$.
\end{theorem}
\begin{proof}
	($\Rightarrow$) Let $a_x \in 1_K$. If $a=0$, then there is nothing to show. Hence suppose that $a>0$. Then, ${1_K}(x) \geq a>0$, that is, $x \in K$. Since $K$ is a subgroup of $G$, $x^{-1} \in K$. Next, since $H$ is a pronormal subgroup of $K$, there exists $y \in \langle H, H^{x^{-1}} \rangle$ such that $H^y = H^{x^{-1}}$. We show that $a_{y^{-1}} \in \langle 1_H, {1_H}^{a_x} \rangle$ and ${1_H}^{a_{y^{-1}}} = {1_H}^{a_x}$.
	
	\noindent Firstly, since $y \in \langle H, H^{x^{-1}} \rangle$, $y^{-1} \in \langle H, H^{x^{-1}} \rangle$. Therefore,
	\[ y^{-1} = {y_1}{y_2}\ldots{y_n}, \text{ where } {y_i} \text{ or } {y_i}^{-1} \in H \cup H^{x^{-1}}. \]
	Note that if $y_i \in H$, then $1_H(y_i) = 1 \geq a$ and hence $a_{y_i} \in 1_H$. On the other hand, if $y_i \in H^{x^{-1}}$, then $1_H(xy_ix^{-1})=1$. Thus	${1_H}^{a_x}(y_i) = a \wedge 1_H(xy_ix^{-1}) = a$.	Hence, $a_{y_i} \in ({1_H})^{a_x}$. Therefore, $a_{y_i} \in 1_H \cup ({1_H})^{a_x}$ for all $i=1,2,\ldots,n.$	Thus,
	\[ a_{y^{-1}} = {a_{y_1}} \circ {a_{y_2}} \circ \ldots \circ {a_{y_n}} \in \langle 1_H, ({1_H})^{a_x} \rangle. \]
	Next, let $g \in G$. If $({1_H})^{a_x}(g) = 0$, then $g \notin H^{x^{-1}} = H^y$. Thus $1_H(y^{-1}gy) = 0$ and hence ${1_H}^{a_{y^{-1}}}(g) = a \wedge 1_H(y^{-1}gy) = 0.$ On the other hand, if $({1_H})^{a_x}(g) > 0$, then $g \in H^{x^{-1}}$. Thus $({1_H})^{a_x}(g) = a$. Now $H^{x^{-1}} = H^y$ implies $y^{-1}gy \in H$. Thus 
	\[ ({1_H})^{a_{y^{-1}}}(g) = a \wedge 1_H(y^{-1}gy) = a = ({1_H})^{a_x}(g). \]
	$(\Leftarrow)$ Let $x \in K$. Then, $1_{x^{-1}} \in 1_K$. Thus there exists $a_y \in \langle 1_H, {1_H}^{1_{x^{-1}}} \rangle$ such that ${1_H}^{a_y} = {1_H}^{1_{x^{-1}}}$. We claim that $y^{-1} \in \langle H, H^x \rangle$ and $H^{y^{-1}} = H^x$.
	Firstly, since ${1_H}^{a_y}(e) = {1_H}^{1_{x^{-1}}}$, we must have $a=1$. Now, $1_y \in \langle 1_H, {1_H}^{1_{x^{-1}}} \rangle$ implies $\langle 1_H, {1_H}^{1_{x^{-1}}} \rangle (y^{-1}) = 1$. By Theorem \ref{gen_form},
	\[ \mathop{\vee}\limits_{c \leq 1}{\left\{c \mid y^{-1} \in \left\langle (1_H \cup {1_H}^{1_{x^{-1}}})_c\right\rangle\right\}} = 1. \]
	Hence there exists $c>0$ such that $y^{-1} \in \langle (1_H \cup {1_H}^{1_{x^{-1}}})_c\rangle$. Then,
	\[ y^{-1} = y_1 y_2 \ldots y_n, \]
	where $y_i \text{ or } {y_i}^{-1} \in (1_H \cup {1_H}^{1_{x^{-1}}})_c.$ Thus $(1_H \cup {1_H}^{1_{x^{-1}}})({y_i}^{-1}) \geq c > 0$. This implies ${1_H}(y_i) > 0$ or ${1_H}^{1_{x^{-1}}}(y_i) > 0$. Therefore, either $y_i \in H$ or $y_i \in H^x$. Hence
	\[ y^{-1} = y_1 y_2 \ldots y_n, \text{ where } y_i \text{ or } {y_i}^{-1} \in (H \cup H^x).  \]
	We conclude that $y^{-1} \in \langle H, H^x \rangle$. Next, note that if $g \in H^{y^{-1}}$, then $ygy^{-1} \in H$. Thus ${1_H}^{1_y}(g) = 1$. Since ${1_H}^{1_y} = {1_H}^{1_{x^{-1}}}$, ${1_H}^{1_{x^{-1}}}(g) = 1$. Thus $g \in H^x$. Hence $H^{y^{-1}} \subseteq H^x$. Similar argument shows that $H^x \subseteq H^{y^{-1}}$. Hence $H^{y^{-1}} = H^x$.
\end{proof}

\noindent Now, we provide an example to demonstrate the pronormal $L$-subgroup of an $L$-group.

\begin{example} \label{example1}
	Let $M=\{ l,f,a,b,c,d,u \}$ be the lattice given by the following figure: 
	\begin{center}
		\includegraphics[scale=1.0]{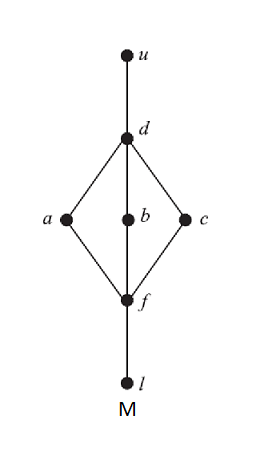}\\		
	\end{center}
	\noindent Let $G=S_4$, the group of all permutations of the set $\{1,2,3,4\}$ with identity element $e$.
	
	\noindent Let $D_4^1 = \langle (24), (1234) \rangle,$ $D_4^2 = \langle (12), (1324) \rangle,$ $D_4^3 = \langle (23), (1342) \rangle$ denote the dihedral subgroups of $G$ and $V_4 = \{e, (12)(34), (13)(24), (14)(23)\}$ denote the Klein-4 subgroup of $G$. 
	
	Define the $L$-subset $\mu$ of $G$ as follows:
	\[ \mu(x) = \begin{cases}
		u &\text{if } x \in V_4,\\
		d &\text{if } x \in S_4 \setminus V_4.
	\end{cases} \]
	Since $\mu_t$ is a subgroup of $G$ for all $t \leq u$, by Theorem \ref{lev_gp}, $\mu \in L(G)$. Next, let $\eta$ be the $L$-subset of $\mu$ be defined by 
	\[ \eta(x) = \begin{cases}
		u &\text{if } x=e,\\
		d &\text{if } x \in V_4 \setminus \{e\},\\
		a &\text{if } x \in D_4^1 \setminus V_4,\\
		b &\text{if } x \in D_4^2 \setminus V_4,\\
		c &\text{if } x \in D_4^3 \setminus V_4,\\
		f &\text{if } x \in S_4 \setminus \mathop{\cup}\limits_{i=1}^3 D_4^i.
	\end{cases} \]
	Since $\eta_t$ is a subgroup of $\mu_t$ for all $t \leq u$, by Theorem \ref{lev_sgp}, $\eta$ is an $L$-subgroup of $\mu$. Now, it can be easily verified that $\eta$ is a pronormal $L$-subgroup of $\mu$. For instance, consider the $L$-point $d_{(123)} \in \mu$. Then,
	\begin{equation*}
		\begin{split}
			\eta^{d_{(123)}}(x) &= d \wedge \eta((123)x(132)) \\
			&= \begin{cases}
				d &\text{if } x \in V_4,\\
				a &\text{if } x \in D_4^3 \setminus V_4,\\
				b &\text{if } x \in D_4^1 \setminus V_4,\\
				c &\text{if } x \in D_4^2 \setminus V_4,\\
				f &\text{if } x \in S_4 \setminus \mathop{\cup}\limits_{i=1}^3 D_4^i.
			\end{cases} 
		\end{split}
	\end{equation*}
	Thus 
	\begin{equation*}
		\begin{split}
			(\eta \cup \eta^{d_{(123)}})(x) = \begin{cases}
				u &\text{if } x =e,\\
				d &\text{if } x \in \mathop{\cup}\limits_{i=1}^3 D_4^i,\\
				f &\text{if } x \in S_4 \setminus \mathop{\cup}\limits_{i=1}^3 D_4^i.
			\end{cases} 
		\end{split}
	\end{equation*}
	Hence 
	\begin{equation*}
		\begin{split}
			\langle \eta \cup \eta^{d_{(123)}} \rangle(x) = \begin{cases}
				u &\text{if } x =e,\\
				d &\text{if } x \in S_4.
			\end{cases} 
		\end{split}
	\end{equation*}
	From this, we see that $d_{(123)} \in \langle \eta \cup \eta^{d_{(123)}} \rangle$. Similarly, for the $L$-point $u_{(12)(34)} \in \mu$, we can see that 
	\[ \eta^{u_{(12)(34)}} = u \wedge \eta((12)(34)x(12)(34)) = \eta. \]
	Hence $\langle \eta, \eta^{u_{(12)(34)}} \rangle = \eta$. Now, 
	\[ u_{(12)(34)} \notin \langle \eta, \eta^{u_{(12)(34)}} \rangle, \] 
	however, $u_e \in \langle \eta, \eta^{u_{(12)(34)}} \rangle$ such that $\eta^{u_e} = \eta^{u_{(12)(34)}}$. Proceeding in a similar manner, pronormality of $\eta$ follows.   
\end{example}

\noindent In Theorem \ref{hom_prn}, we discuss the image of a pronormal $L$-subgroup under group homomorphisms. For this, we recall Lemma \ref{hom_conj} from \cite{jahan_conj}.

\begin{lemma}\label{hom_conj} (\cite{jahan_conj})
	Let $f : G \rightarrow H$ be a group homomorphism and $\mu \in L(G)$. Then, for $\eta \in L(\mu)$ and $a_z \in \mu$, the $L$-subgroup $f(\eta^{a_z})$ is a conjugate $L$-subgroup of $f(\eta)$ in $f(\mu)$. In fact, 
	\[ f(\eta^{a_z}) = f(\eta)^{a_{f(z)}}. \]
\end{lemma}

\begin{lemma}\label{hom_lev}
	Let $f:G \rightarrow H$ be a group homomorphism and $\mu \in L(G)$. Then, for $\eta \in L(\mu)$, 
	\[ f(\eta_t) \subseteq f(\eta)_t \]
	for all $t \leq \eta(e)$.
\end{lemma}

\begin{theorem}\label{hom_prn}
	Let $f : G \rightarrow H$ be a surjective group homomorphism. Let $\mu \in L(G)$ such that $\mu$ possesses sup-property. If $\eta$ is a pronormal $L$-subgroup of $\mu$, then $f(\eta)$ is a pronormal $L$-subgroup of f($\mu$).
\end{theorem}
\begin{proof}
	Let $a_x \in f(\mu)$. We have to show that there exists an $L$-point $b_y \in \langle f(\eta), f(\eta)^{a_x} \rangle$ such that $f(\eta)^{b_y} = f(\eta)^{a_x}$.	Since $a_x \in f(\mu)$, $f(\mu)(x) \geq a$. By definition,
	\[ f(\mu)(x) = \vee \{ \mu(g) \mid g \in f^{-1}(x) \}. \]
	Let $A = \{ g \in G \mid g \in f^{-1}(x) \}$. Since $f$ is a surjection, $A$ is a non-empty subset of $G$. Since $\mu$ possesses the sup-property, there exists $s \in A$ such that 
	\[ a \leq f(\mu)(x) = \vee \{ \mu(g) \mid g \in A \} = \mu(s). \] 
	Hence $f(s) = x$ and $a_s \in \mu$. Now, since $\eta$ is a pronormal $L$-subgroup of $\mu$, there exists $b_t \in \langle \eta, \eta^{a_s} \rangle$ such that $\eta^{b_t} = \eta^{a_s}$. We claim that $b_{f(t)}$ is the required $L$-point.
	
	\noindent Firstly, we show that $b_{f(t)} \in \langle f(\eta), f(\eta)^{a_x} \rangle$. Since $b_t \in \langle \eta, \eta^{a_s} \rangle$, $\langle \eta, \eta^{a_s} \rangle(t) \geq b$. By Theorem \ref{gen_form},
	\[  \langle \eta, \eta^{a_s} \rangle(t) = \mathop{\vee}\limits_{c \leq \eta(e)}\left\{c \mid t \in \langle(\eta \cup \eta^{a_s})_c \rangle \right\} \]
	Let $c \leq \eta(e)$ such that $t \in \langle (\eta \cup \eta^{a_s})_c \rangle$. Then,
	\[ t = t_1 t_2 \ldots t_n, \text{ where } t_i \text{ or } {t_i}^{-1} \in (\eta \cup \eta^{a_s})_c. \]
	This implies
	\[ f(t) = f(t_1)f(t_2)\ldots f(t_n),\] 
	where $f(t_i) \text{ or } f(t_i)^{-1} \in f((\eta \cup \eta^{a_s})_c)$.
	By Lemma \ref{hom_lev}, $f((\eta \cup \eta^{a_s})_c) \subseteq (f(\eta \cup \eta^{a_s}))_c = (f(\eta) \cup f(\eta^{a_s}))_c$. Also, by Theorem \ref{hom_conj}, $(f(\eta^{a_s})) = f(\eta)^{a_{f(s)}} = f(\eta)^{a_x}$. Hence,
	\[ f(t) = f(t_1)f(t_2)\ldots f(t_n),\] 
	where $f(t_i)$ $f(t_i)^{-1} \in (f(\eta) \cup f(\eta)^{a_x})_c$, that is, $f(t) \in \langle f(\eta) \cup f(\eta^{a_x}))_c \rangle$. Thus
	\begin{equation*}
	\begin{split}
	\hspace{2em}&\hspace{-2em} \langle f(\eta), f(\eta)^{a_x} \rangle (f(t)) \\
			&= \mathop{\vee}\limits_{c \leq f(\eta)(e)}\left\{c \mid f(t) \in \langle(f(\eta) \cup f(\eta)^{a_x})_c \rangle \right\}\\
			&\geq \mathop{\vee}\limits_{c \leq \eta(e)}\left\{c \mid t \in \langle(\eta \cup \eta^{a_s})_c \rangle \right\}\\
			&= \langle \eta, \eta^{a_s} \rangle (t)\\
			&\geq b.  
	\end{split}
	\end{equation*}
	Hence $b_{f(t)} \in \langle f(\eta), f(\eta)^{a_x} \rangle$. Next, since $\eta^{a_s} = \eta^{b_t}$, by Lemma \ref{hom_conj},
	\[ f(\eta)^{a_x} = f(\eta^{a_s}) = f(\eta^{b_t}) = f(\eta)^{b_{f(t)}}. \] 
	Hence $b_{f(t)}$ is the required $L$-point and we conclude that $f(\eta)$ is a pronormal $L$-subgroup of $f(\mu)$.
\end{proof}

\noindent Below, we provide a level subset characterization for pronormal $L$-subgroups. For this, we recall that a lattice $L$ is said to be upper well ordered if every non-empty subset of $L$ contains its supremum. 

\begin{theorem}
	\label{lev_prn}
	Let $L$ be an upper well ordered lattice and $\mu \in L(G)$. If $\eta$ is a pronormal $L$-subgroup of $\mu$, then $\eta_t$ is a pronormal subgroup of $\mu_t$ for all $t \leq \eta(e)$.
\end{theorem}
\begin{proof}
	Let $\eta$ be a pronormal $L$-subgroup of $\mu$ and let $t \leq \eta(e)$. To show that $\eta_t$ is a pronormal subgroup of $\mu_t$, let $g \in \mu_t$. Then, $t_{g^{-1}} \in \mu$. Hence there exists an $L$-point $a_x \in \langle \eta, \eta^{t_{g^{-1}}} \rangle$ such that $\eta^{a_x} = \eta^{t_{g^{-1}}}$. We claim that $x^{-1} \in \langle \eta_t, {\eta_t}^g \rangle$ and ${\eta_t}^{x^{-1}} = {\eta_t}^g$.
	
	\noindent Firstly, since $\eta^{a_x} = \eta^{t_{g^{-1}}}$, 
	\[ a \geq a \wedge \eta(e) = \text{tip}(\eta^{a_x}) = \text{tip}(\eta^{t_{g^{-1}}}) = t \wedge \eta(e) = t. \]
	Hence $a \geq t$. Now, we show that $x^{-1} \in \langle \eta_t, {\eta_t}^g \rangle$. Note that since $a_x \in \langle \eta, \eta^{t_{g^{-1}}} \rangle$,
	\[ \langle \eta, \eta^{t_{g^{-1}}} \rangle (x) \geq a \geq t. \]
	By Theorem \ref{gen_form},
	\begin{center} 
		$\langle \eta, \eta^{t_{g^{-1}}} \rangle(x) = \mathop{\vee}\limits_{c \leq \eta(e)}{\left\{c \mid x\in\left\langle (\eta \cup \eta^{t_{g^{-1}}})_c \right\rangle\right\}}.$ 
	\end{center}
	Let $A = \{ c \leq \eta(e) \mid x\in\left\langle (\eta \cup \eta^{t_{g^{-1}}})_c \right\rangle \}$. Then, $A$ is a non-empty subset of $L$. Since $L$ is upper well ordered, $A$ contains its supremum, say $c_0$. Thus $x \in \langle (\eta \cup \eta^{t_g^{-1}})_{c_0} \rangle$ and $c_0 \geq t$. This implies $(\eta \cup \eta^{t_{g^{-1}}})_{c_0} \subseteq (\eta \cup \eta^{t_{g^{-1}}})_t$ and hence $x \in \langle (\eta \cup \eta^{t_{g^{-1}}})_t \rangle$. Thus
	\[ x = x_1 x_2 \ldots x_k, \text{ where } {x_i} \text{ or } {x_i}^{-1} \in (\eta \cup \eta^{t_{g^{-1}}})_t, \]
	that is, $(\eta \cup \eta^{t_{g^{-1}}})(x_i) \geq t$. This implies 
	\[ \eta(x_i) \vee \eta^{t_{g^{-1}}}(x_i) \geq t. \]
	Again, since $L$ is upper well ordered, $\eta(x_i) \geq t$ or $\eta^{t_{g^{-1}}}(x_i) \geq t$. If $\eta(x_i) \geq t$, then $x_i \in \eta_t$. On the other hand, if $\eta^{t_{g^{-1}}}(x_i) \geq t$, then 
	\[ \eta(g^{-1} x_i g) \geq t \wedge \eta(g^{-1} x_i g) \geq t, \]
	that is, $x_i \in {\eta_t}^g$. Thus $x_i \in \eta_t \cup {\eta_t}^g$. Therefore
	\[ x= x_1 x_2 \ldots x_k, \text{ where } x_i \text{ or } {x_i}^{-1} \in \eta_t \cup {\eta_t}^g. \]
	This implies $x \in \langle \eta_t, {\eta_t}^g \rangle$. Since $\langle \eta_t, {\eta_t}^g \rangle$ is a subgroup of $G$, we conclude that $x^{-1} \in \langle \eta_t, {\eta_t}^g \rangle$.
	
	\noindent Finally, we show that ${\eta_t}^{x^{-1}} = {\eta_t}^g$. Let $z \in {\eta_t}^{x^{-1}}$ be arbitrary. Then, $xzx^{-1} \in \eta_t$, that is,	$\eta(xzx^{-1}) \geq t$. This implies 
	\[ a \wedge \eta(xzx^{-1}) \geq a \wedge t = t. \]
	Thus $\eta^{a_x}(z) \geq t$. Since $\eta^{a_x} = \eta^{t_{g^{-1}}}$, $\eta^{t_{g^{-1}}}(z) \geq t$. Hence 
	\[ t \wedge \eta(g^{-1}zg) \geq t. \]
	It follows that	$\eta(g^{-1}zg) \geq t.$ Therefore $g^{-1}zg \in \eta_t$, that is, $z \in {\eta_t}^g$. Thus ${\eta_t}^{x^{-1}} \subseteq {\eta_t}^g$.
	
	\noindent For the reverse inclusion, let $z \in {\eta_t}^g$. Then, $g^{-1}zg \in \eta_t$, that is, $\eta(g^{-1}zg) \geq t.$  Thus $t \wedge \eta(g^{-1}zg) \geq t$, or $\eta^{t_{g^{-1}}}(z) \geq t.$ Since $\eta^{t_{g^{-1}}} = \eta^{a_x}$, $\eta^{a_z}(x) \geq t$. Hence
	\[ a \wedge \eta(xzx^{-1}) \geq t. \]
	This implies $\eta(xzx^{-1}) \geq t$, that is, $xzx^{-1} \in \eta_t$.Thus $z \in {\eta_t}^{x^{-1}}$. Since $z$ is an arbitrary element of ${\eta_t}^g$, it follows that ${\eta_t}^g \subseteq {\eta_t}^{x^{-1}}$. We conclude that ${\eta_t}^{x^{-1}} = {\eta_t}^g$. This completes the proof.	
\end{proof}

\section{Pronormal $L$-Subgroups and Normality}

\noindent In this section, we explore the various relations of pronormal $L$-subgroups with the notions of normal $L$-subgroups, subnormal $L$-subgroups, normalizer of an $L$-subgroup of an $L$-group and maximal $L$-subgroup. The results discussed in this section parallel the interactions of these concepts in classical group theory. Hence this section highlights the strengths of the notion of pronormality developed in this paper.

In Theorem \ref{nor_prn}, we prove that a normal $L$-subgroup of an $L$-group $\mu$ is pronormal in $\mu$. Firstly, we recall the following result from \cite{jahan_conj}:

\begin{lemma}(\cite{jahan_conj})
	\label{nor_conj}
	Let $\eta \in L(\mu)$. Then, $\eta$ is a normal $L$-subgroup of $\mu$ if and only if  $\eta^{a_z} \subseteq \eta$ for every $L$-point $a_z \in \mu$. Moreover, if $\eta \in NL(\mu)$ and $\text{tip}(\eta^{a_z}) = \text{tip}(\eta)$, then $\eta^{a_z} = \eta$. 
\end{lemma}	

\begin{theorem}\label{nor_prn}
	Let $\eta$ be a normal $L$-subgroup of $\mu$. Then, $\eta$ is a pronormal $L$-subgroup of $\mu$.
\end{theorem}
\begin{proof}
	Let $\eta$ be a normal $L$-subgroup of $\mu$ and let $a_x \in \mu$. Note that by Lemma \ref{nor_conj},	$\eta^{a_x} \subseteq \eta$. Thus 
	\[ \langle \eta, \eta^{a_x} \rangle = \eta.\]
	Take $b=a \wedge \eta(e)$. Then, $b_e \in \eta=\langle \eta, \eta^{a_x} \rangle$. We claim that	$\eta^{b_e} = \eta^{a_x}.$
	
	\noindent Let $g \in G$. Then, 
	\begin{equation*}
		\begin{split}
			\eta^{a_x}(g) &= a \wedge \eta(xgx^{-1}) \\
			&\geq a\wedge \eta(g) \wedge \mu(x) \qquad~~~ (\text{since $\eta \in NL(\mu)$})\\
			&= a \wedge \eta(g) \qquad\qquad\qquad~~~ (\text{since $\mu(x)\geq a$})\\
			&= a \wedge \eta(e) \wedge \eta(g) \qquad~~~ (\text{since $\eta(e) \geq \eta(g)$})\\
			&= b \wedge \eta(ege^{-1})\\
			&= \eta^{b_e}(g).
		\end{split}
	\end{equation*}
	Hence $\eta^{b_e} \subseteq \eta^{a_x}$. For the reverse inclusion,
	\begin{equation*}
		\begin{split}
			\eta^{b_e}(g) &= b \wedge \eta(ege^{-1})\\ 
			&= \{a \wedge \eta(e)\} \wedge \eta(g)\\
			&= a \wedge \eta(e) \wedge \eta(x^{-1}(xgx^{-1})x)\\
			& \geq a \wedge \eta(e) \wedge \eta(xgx^{-1}) \wedge \mu(x)\\ 
			&\qquad\qquad\qquad\qquad (\text{since $\eta$ is normal in $\mu$})\\
			&= a \wedge \eta(gxg^{-1})\\ 
			&\qquad~ (\text{since $\mu(x) \geq a$ and $\eta(e) \geq \eta(xgx^{-1})$})\\
			&= \eta^{a_x}(g).
		\end{split}
	\end{equation*}
	We conclude that $\eta^{a_x} = \eta^{b_e}$ and hence $\eta$ is a pronormal $L$-subgroup of $\mu$.
\end{proof}

\begin{example}
	Consider the $L$-subgroup $\eta$ of the $L$-group $\mu$ discussed in Example \ref{example1}. We have already shown that $\eta$ is a pronormal $L$-subgroup of $\mu$. Here, note that for $t=a$, $\eta_a = D_4^1$, which is not a normal subgroup of $\mu_a = S_4$. Hence by Thorem \ref{lev_norsgp}, $\eta \notin NL(\mu)$.
\end{example}

\noindent The normalizer of an $L$-subgroup has been explored in detail by Ajmal and Jahan in \cite{ajmal_nor}. Therein, they have defined the normalizer using the notion of cosets of $L$-subgroups. The normalizer thus developed has been shown to be immensely compatible with the notion of normality in $L$-group theory.

In Theorem \ref{prn_norm}, we discuss the pronormality of the normalizer of a pronormal $L$-subgroup of an $L$-group. Firstly, we recall the definitions of cosets and normalizer from \cite{ajmal_nor}:  

\begin{definition}(\cite{ajmal_nor})
	Let $\eta \in L(\mu)$ and let $a_x$ be an $L$-point of $\mu$. The left (respectively, right) coset of $\eta$ in $\mu$ with respect to $a_x$ is defined as the set product $a_x \circ \eta$ ($\eta \circ a_x$).
\end{definition}

\noindent From the definition of set product of two $L$-subsets, it can be easily seen that for all $z \in G$, 
\[ (a_x \circ \eta)(z) = a \wedge \eta(x^{-1}z) \] 
and
\[ (\eta \circ a_x)(z) = a \wedge \eta(zx^{-1}). \]

\begin{definition}(\cite{ajmal_nor})
	\label{defn_norm1}
	Let $\eta \in L(\mu)$. The normalizer of $\eta$ in $\mu$, denoted by $N(\eta$), is the $L$-subgroup defined as follows:
	\[ N(\eta) = \bigcup \left\{ a_x \in \mu \mid a_x \circ \eta = \eta \circ a_x \right\}. \]
	$N(\eta)$ is the largest $L$-subgroup of $\mu$ such that $\eta$ is a normal $L$-subgroup of $N(\eta)$. Also, it has been established in \cite {ajmal_nor} that  $\eta$ is a normal $L$-subgroup of $\mu$ if and only if $N(\eta)=\mu$.
\end{definition}

\noindent In \cite{jahan_conj}, the authors have provided a new definition for the normalizer of an $L$-subgroup using the notion of the conjugate. We recall this definition as a theorem below: 

\begin{theorem}[\cite{jahan_conj}]\label{def_norm}
	Let $\eta \in L(\mu)$. The normalizer of $\eta$ in $\mu$, denoted by $N(\eta$), is the $L$-subgroup defined as follows:
	\[ N(\eta) = \bigcup \left\{ a_z \in \mu \mid  \eta^{a_z} \subseteq \eta \right\}. \]
\end{theorem}

\begin{theorem}\label{prn_norm}
	Let $\eta$ be a pronormal $L$-subgroup of $\mu$ satisfying $\text{tip}(\eta)=\text{tip}(\mu)$. Let $N(\eta)$ denote the normalizer of $\eta$ in $\mu$. Then, $N(\eta)$ is a pronormal $L$-subgroup of $\mu$.
\end{theorem}
\begin{proof}
	Let $\nu = N(\eta)$. Let $a_x$ be an $L$-point of $\mu$. Then, since $\eta$ is a pronormal $L$-subgroup of $\mu$, there exists $b_y \in \langle \eta, \eta^{a_x} \rangle$ such that $\eta^{b_y} = \eta^{a_x}$. We claim that $(a \wedge b)_x \in \langle \nu, \nu^{a_x} \rangle$ and $\nu^{(a \wedge b)_x} = \nu^{a_x}$.
	
	\noindent Firstly, note that for all $g \in G$,
	\begin{equation*}
		\begin{split}
			\eta^{(a \wedge b)_{xy^{-1}}}(g) &= (a \wedge b) \wedge \eta((xy^{-1})g(xy^{-1})^{-1}) \\
			&= b \wedge (a \wedge \eta(x(y^{-1}gy)x^{-1})) \\
			&= b \wedge \eta^{a_x}(y^{-1}gy)\\
			&= b \wedge \eta^{b_y}(y^{-1}gy) \\ 
			&\qquad\qquad\qquad\qquad (\text{since } \eta^{a_x} = \eta^{b_y})\\
			&= b \wedge \eta(g)\\
			&\leq \eta(g).
		\end{split}
	\end{equation*} 
	Thus $\eta^{(a \wedge b)_{xy^{-1}}} \subseteq \eta$. By Theorem \ref{def_norm}, $(a \wedge b)_{xy^{-1}} \in N(\eta) = \nu$. Thus $(a \wedge b)_{xy^{-1}} \in \langle \nu, \nu^{a_x} \rangle$. Also, $b_y \in \langle \eta, \eta^{a_x} \rangle \subseteq \langle \nu, \nu^{a_x} \rangle$. Therefore
	\[ (a \wedge b)_x = (a \wedge b)_{xy^{-1}} \circ b_y \subseteq \langle \nu, \nu^{a_x} \rangle. \]
	Now, we show that $\nu^{(a \wedge b)_x} = \nu^{a_x}$. Firstly, since $\text{tip}(\eta) = \text{tip}(\mu)$ and $\eta \subseteq N(\eta) \subseteq \mu$, we must have $\text{tip}(N(\eta)) = \text{tip}(\eta)$. Moreover, since $\eta^{a_x} = \eta^{b_y}$, $\text{tip}(\eta^{a_x}) = \text{tip}(\eta^{b_y})$, that is, $a \wedge \eta(e) = b \wedge \eta(e)$. Thus $a \wedge \nu(e) = b \wedge \nu(e)$. Hence for all $g \in G$,
	\begin{equation*}
		\begin{split}
			\nu^{(a \wedge b)_x}(g) &= (a \wedge b) \wedge \nu(gxg^{-1}) \\
			&= (a \wedge b) \wedge (\nu(e) \wedge \nu(xgx^{-1})) \\
			&\qquad\qquad\qquad (\text{since } \nu(e) \geq \nu(gxg^{-1}))\\
			&= a \wedge (b \wedge \nu(e)) \wedge \nu(xgx^{-1}) \\
			&= a \wedge (a \wedge \nu(e)) \wedge \nu(xgx^{-1}) \\
			&= a \wedge (\nu(e) \wedge \nu(xgx^{-1})) \\
			&= a \wedge \nu(xgx^{-1}) \\
			&= \nu^{a_x}(g).
		\end{split}
	\end{equation*}
	Therefore $\nu^{(a \wedge b)_x} = \nu^{a_x}$ and we conclude that $\nu = N(\eta)$ is a pronormal $L$-subgroup of $\mu$.
\end{proof}

\noindent In Theorem \ref{prn_setp}, we show that the set product of a normal $L$-subgroup and a pronormal $L$-subgroup of $\mu$ is a pronormal $L$-subgroup of $\mu$. For this, we recall the following from \cite{jahan_conj}:

\begin{lemma}(\cite{jahan_conj})
	\label{conj_prod}
	Let $\eta, \nu \in L(\mu)$ and $a_z$ be an $L$-point of $\mu$. Then, 
	\[ (\eta \circ \nu)^{a_z} = \eta^{a_z} \circ \nu^{a_z}. \]
\end{lemma}

\begin{theorem} \label{prn_setp}
	Let $\eta$ be a normal $L$-subgroup of $\mu$ and $\nu$ be a pronormal $L$-subgroup of $\mu$ such that $\text{tip}(\eta) = \text{tip}(\nu)$. Then, $\eta \circ \nu$ is a pronormal $L$-subgroup of $\mu$.
\end{theorem}
\begin{proof}
	Since $\eta$ is a normal $L$-subgroup of $\mu$, $\eta \circ \nu$ is an $L$-subgroup of $\mu$. To show that $\eta \circ \nu$ is a pronormal $L$-subgroup of $\mu$, let $a_x \in \mu$. Then, since $\eta$ is a normal $L$-subgroup of $\mu$, by Theorem \ref{nor_prn}, $\eta$ is a pronormal $L$-subgroup of $\mu$. In particular, the $L$-point $b_e$ of $\mu$, where $b = a \wedge \eta(e)$, satisfies $b_e \in \langle \eta, \eta^{a_x} \rangle$ and $\eta^{b_e} = \eta^{a_x}$. Now, clearly $(a \wedge b)_x \in \mu$. Hence by pronormality of $\nu$ in $\mu$, there exists an $L$-point $c_y$ such that $c_y \in \langle \nu, \nu^{(a \wedge b)_x} \rangle$ and $\nu^{c_y} = \nu^{(a \wedge b)_x}$. Consider the $L$-point $(b \wedge c)_y \in \mu$. We claim that $(b \wedge c)_y \in \langle \eta \circ \nu, (\eta \circ \nu)^{a_x} \rangle$ and $(\eta \circ \nu)^{(b \wedge c)_y} = (\eta \circ \nu)^{a_x}$.
	
	\noindent Firstly, we show that $(b \wedge c)_y \in \langle \eta \circ \nu, (\eta \circ \nu)^{a_x} \rangle$. Note that since $\eta \subseteq \eta \circ \nu$ and $\eta^{a_x} \subseteq \eta^{a_x} \circ \nu^{a_x} = (\eta \circ \nu)^{a_x}$,
	\[ b_e \in \langle \eta \circ \nu, (\eta \circ \nu)^{a_x} \rangle. \]
	Now, $\nu^{a_x} \subseteq (\eta \circ \nu)^{a_x} \subseteq \langle \eta \circ \nu , (\eta \circ \nu)^{a_x} \rangle$. Hence
	\begin{equation*}
	\begin{split}
		\nu^{(a \wedge b)_x} = (\nu^{a_x})^{b_e} &\subseteq (\langle \eta \circ \nu, (\eta \circ \nu)^{a_x} \rangle)^{b_e}\\ 
		&\subseteq \langle \eta \circ \nu, (\eta \circ \nu)^{a_x} \rangle.
	\end{split}
	\end{equation*} 
	By assumption, $c_y \in \langle \nu, \nu^{(a \wedge b)_x} \rangle$. Moreover, since $\nu \subseteq \langle \eta \circ \nu, (\eta \circ \nu)^{a_x} \rangle$ and $\nu^{(a \wedge b)_x} \subseteq \langle \eta \circ \nu, (\eta \circ \nu)^{a_x} \rangle$, 
	\[ \langle \nu, \nu^{(a \wedge b)_x} \rangle \subseteq \langle \eta \circ \nu, (\eta \circ \nu)^{a_x} \rangle. \]
	Hence $c_y \in \langle \eta \circ \nu, (\eta \circ \nu)^{a_x} \rangle$. Therefore
	\[ (b \wedge c)_y = b_e \circ c_y \subseteq \langle \eta \circ \nu, (\eta \circ \nu)^{a_x} \rangle. \]	
	Next, we show that $(\eta \circ \nu)^{(b \wedge c)_y} = (\eta \circ \nu)^{a_x}$. Here, note that 
	\[ \eta^{(b \wedge c)_y} = (\eta^{b_e})^{c_y} = (\eta^{a_x})^{c_y}. \]
	Now, $\eta^{a_x}$ is a normal $L$-subgroup of $\mu$, since for all $g, h \in G$,
	\begin{equation*}
		\begin{split}
			\eta^{a_x}(ghg^{-1}) &= a \wedge \eta(x(ghg^{-1})x^{-1}) \\
			&= a \wedge \eta((xgx^{-1})(xhx^{-1})(xgx^{-1})^{-1}) \\
			&\geq a \wedge \eta(xhx^{-1}) \wedge \mu(xgx^{-1})\\ 
			&\qquad\qquad ~~~~~~~~(\text{since } \eta \in NL(\mu)) \\
			&\geq a \wedge \eta(xhx^{-1}) \wedge \mu(x) \wedge \mu(g) \\
			&~~~~~~\qquad\qquad (\text{since } \mu \in L(G)) \\
			&= a \wedge \eta(xhx^{-1}) \wedge \mu(g) \\ 
			&\qquad\qquad\qquad\qquad~~~ (\text{since } a_x \in \mu) \\
			&= \eta^{a_x}(h) \wedge \mu(g). 
		\end{split}
	\end{equation*}
	Moreover, since $\nu^{c_y} = \nu^{(a \wedge b)_x}$ and $\text{tip}(\eta) = \text{tip}(\nu)$, 
	\[c \geq c \wedge \text{tip}(\eta) = (a \wedge b) \wedge \text{tip}(\eta) = a \wedge \text{tip}(\eta) = \text{tip}(\eta^{a_x}). \]
	Hence by Lemma \ref{nor_conj}, $(\eta^{a_x})^{c_y} = \eta^{a_x}$. Therefore $\eta^{(b \wedge c)_y} = \eta^{a_x}$. Also,
	\[ \nu^{(b \wedge c)_y} = (\nu^{c_y})^{b_e} = (\nu^{(a \wedge b)_x})^{b_e} = \nu^{(a \wedge b)_x}.\]
	Here, it is easy to see that $\nu^{(a \wedge b)_x} = \nu^{a_x}$, since for all $g \in G$,
	\begin{equation*}
		\begin{split}
			\nu^{(a \wedge b)_x}(g) &= (a \wedge b) \wedge \nu(xgx^{-1}) \\
			&= a \wedge (a \wedge \nu(e)) \wedge \nu(xgx^{-1}) \\
			&= a \wedge \nu(xgx^{-1}) \\
			&= \nu^{a_x}(g). 
		\end{split}
	\end{equation*}
	Hence $\nu^{(b \wedge c)_y} = \nu^{a_x}$. Hence by Lemma \ref{conj_prod},
	\begin{equation*}
	\begin{split}
		(\eta \circ \nu)^{(b \wedge c)_y} &= (\eta^{(b \wedge c)_y}) \circ (\nu^{(b \wedge c)_y}) \\ 
		&= \eta^{a_x} \circ \nu^{a_x} \\
		&= (\eta \circ \nu)^{a_x}.
	\end{split}
	\end{equation*}
	Thus we conclude that $\eta \circ \nu$ is a pronormal $L$-subgroup of $\mu$.
\end{proof}

\noindent Next, we show that every maximal $L$-subgroup of an $L$-group $\mu$ is a pronormal $L$-subgroup of $\mu$. For this, we recall the definition of maximal $L$-subgroups from \cite{jahan_max}:

\begin{definition}(\cite{jahan_max})
	Let $\mu \in L(G)$. A proper $L$-subgroup $\eta$ of $\mu$ is said to be a maximal $L$-subgroup of $\mu$ if, whenever $\eta \subseteq \theta \subseteq \mu$ for some $\theta \in L(\mu)$, then either $\theta = \eta$ or $\theta = \mu$.
\end{definition}

\begin{theorem}
	Let $\eta$ be a maximal $L$-subgroup of $\mu$. Then, $\eta$ is a pronormal $L$-subgroup of $\mu$.
\end{theorem}
\begin{proof}
	Let $\eta$ be a maximal $L$-subgroup of $\mu$ and let $N(\eta)$ denote the normalizer of $\eta$ in $\mu$. Then,
	\[ \eta \subseteq N(\eta) \subseteq \mu. \]
	By maximality of $\eta$, either $N(\eta) = \mu$ or $N(\eta) = \eta$. If $N(\eta) = \mu$, then $\eta$ is a normal $L$-subgroup of $\mu$. Hence by Theorem \ref{nor_prn}, $\eta$ is a pronormal $L$-subgroup of $\mu$ and we are done. Now, suppose that $N(\eta) = \eta$. Let $a_x \in \mu$. We show that $a_x \in \langle \eta, \eta^{a_x} \rangle$. We have the following cases:
	\begin{description}
		\item[Case 1: $\mathbf{\eta^{a_x} \nsubseteq \eta.}$]	Then, $\eta \subsetneq \langle \eta, \eta^{a_x} \rangle \subseteq \mu$. By maximality of $\eta$, $\langle \eta, \eta^{a_x} \rangle = \mu$. Thus $a_x \in \langle \eta, \eta^{a_x} \rangle$.
		\item[Case 2: $\mathbf{\eta^{a_x} \subseteq \eta}$.] Then, by Theorem \ref{def_norm}, $a_x \in N(\eta)$. Since $N(\eta) = \eta$, $a_x \in \langle \eta, \eta^{a_x} \rangle$.
	\end{description}
	\noindent Hence in both the cases, $a_x \in \langle \eta, \eta^{a_x} \rangle$. Thus we conclude that $\eta$ is a pronormal $L$-subgroup of $\mu$.
\end{proof}

\noindent In Theorem \ref{subprn}, we present the main result of this paper: an $L$-subgroup of $\mu$ that is both pronormal and subnormal in $\mu$ is a normal $L$-subgroup of $\mu$. This result shows that the concept of pronormality introduced in this study is agreeable with these notions, like their classical counterparts. The notion of subnormal $L$-subgroups was introduced in \cite{ajmal_nc} and studied in detail in \cite{ajmal_subnormal}. Below, we recall the definition of normal closure and subnormality from \cite{ajmal_nc}:

\begin{definition}(\cite{ajmal_nc})
	Let $\eta \in L(\mu)$. The $L$-subset $\mu\eta\mu^{-1}$ of $\mu$ defined by 
	\[ \mu\eta\mu^{-1}(x) = \bigvee_{x=zyz^{-1}} \left\{ \eta(y) \wedge \mu(z) \right\} ~~\text{ for each } x \in G \]
	is called the conjugate of $\eta$ in $\mu$.	The normal closure of $\eta$ in $\mu$, denoted by $\eta^\mu$, is defined to be the $L$-subgroup of $\mu$ generated by the conjugate $\mu\eta\mu^{-1}$, that is,
	\[ \eta^\mu = \langle \mu\eta\mu^{-1} \rangle. \]
	Moreover, $\eta^\mu$ is the smallest normal $L$-subgroup of $\mu$ containing $\eta$.
\end{definition}

\begin{definition}(\cite{ajmal_nc})
	Let $\eta \in L(\mu)$. Define a descending series of $L$-subgroups of $\mu$ inductively as follows:
	\[ \eta_0 = \mu \qquad \text{ and } \qquad \eta_i = \eta^{\eta_{i-1}} \qquad \text{ for all } i \geq 1. \]
	Then, $\eta_i$ is the smallest normal $L$-subgroup of $\eta_{i-1}$ containing $\eta$, called the $i^{th}$ normal closure of $\eta$ in $\mu$. The series of $L$-subgroups
	\[ \mu = \eta_0 \supseteq \eta_1 \supseteq \ldots \supseteq \eta_{i-1} \supseteq \eta_i \supseteq \ldots \]
	is called the normal closure series of $\eta$ in $\mu$. Moreover, if there exists a non-negative integer $m$ such that 
	\[ \eta = \eta_m \vartriangleleft \eta_{m-1} \vartriangleleft \ldots \vartriangleleft \eta_0 = \mu, \]
	then $\eta$ is said to be a subnormal $L$-subgroup of $\mu$ with defect $m$. 
	
	\noindent Clearly, $m=0$ if $\eta = \mu$ and $m=1$ if $\eta \in NL(\mu)$ and $\eta \neq \mu$.
\end{definition}

\noindent Here, we prove the following:

\begin{lemma}
	\label{con_nc}
	Let $\eta \in L(\mu)$ and $a_z$ be an $L$-point of $\mu$. Then, $\eta^{a_z}$ is contained in the normal closure of $\eta$ in $\mu$.
\end{lemma}
\begin{proof}
	Let $g \in G$. Then,
	\begin{equation*}
		\begin{split}
			\mu\eta\mu^{-1}(g) &= \bigvee_{g=xyx^{-1}} \left\{ \eta(y) \wedge \mu(x) \right\}\\
			&\geq \eta(zgz^{-1}) \wedge \mu(z^{-1})\\
			&\geq \eta(zgz^{-1}) \wedge a \qquad \qquad (\text{since $a_z \in \mu$})\\
			&= \eta^{a_z}(g).
		\end{split}
	\end{equation*}
	Since $g$ is an arbitrary element of $G$, we conclude that
	\[ \eta^{a_z} \subseteq \mu\eta\mu^{-1} \subseteq \eta^\mu. \]	
\end{proof}

\noindent The following result is immediate from the definition of pronormal $L$-subgroups. We state it here without proof.

\begin{lemma}
	\label{prn_subgp}
	Let $\eta$ and $\nu$ be $L$-subgroups of $\mu$ such that $\eta \subseteq \nu$. If $\eta$ is a pronormal $L$-subgroup of $\mu$, then $\eta$ is a pronormal $L$-subgroup of $\nu$.
\end{lemma}

\begin{theorem}{\label{subprn}}
	Let $\eta \in L(\mu)$. If $\eta$ is both a pronormal and subnormal $L$-subgroup of $\mu$, then $\eta$ is a normal $L$-subgroup of $\mu$.
\end{theorem}
\begin{proof}
	Let $\eta$ be a pronormal and subnormal $L$-subgroup of $\mu$ with defect $m \geq 2$. We show that $\eta$ is normal in $\mu$ by applying induction on $m$. 
	
	\noindent Suppose that $\eta$ is subnormal in $\mu$ with defect $2$ and let
	\[ \eta = \eta_2 \lhd \eta_1=\eta^{\mu} \lhd \eta_0 = \mu \]
	be the normal closure series of $\eta$. To show that $\eta$ is normal in $\mu$, let $x,g \in G$. Let $a = \mu(g)$. Then, $a_{g^{-1}} \in \mu$. By Lemma \ref{con_nc}, $\eta^{a_{g^{-1}}} \subseteq \eta^\mu = \eta_{1}$. Now, since $\eta$ is a pronormal $L$-subgroup of $\mu$, there exists $b_w \in \langle \eta, \eta^{a_{g^{-1}}} \rangle \subseteq \eta_1$ such that 
	\[ \eta^{b_w} = \eta^{a_{g^{-1}}}. \] 
	Since $\eta$ is normal in $\eta_1$ and $b_w \in \eta_1$, by Lemma \ref{nor_conj}, $\eta^{b_w} \subseteq \eta$. Hence $\eta^{a_{g^{-1}}} \subseteq \eta$. Therefore
	\[ \eta^{a_{g^{-1}}}(gxg^{-1}) \leq \eta(gxg^{-1}), \]
	that is,
	\[ a \wedge \eta(x) \leq \eta(gxg^{-1}). \]
	Since $a = \mu(g)$,
	\[ \eta(gxg^{-1}) \geq \eta(x) \wedge \mu(g). \]
	Therefore $\eta$ is a normal $L$-subgroup of $\mu$. Hence the result is true for $m=2$.
	
	\noindent Next, suppose that the result holds for $m-1$, that is, if $\eta$ is a pronormal and subnormal $L$-subgroup of subnormal with defect $m-1$, then $\eta$ is a normal $L$-subgroup of $\mu$.
	
	\noindent Suppose that $\eta$ is a pronormal and subnormal $L$-subgroup of $\mu$ with defect $m$. Let
	\[ \eta = \eta_m \lhd \eta_{m-1} \lhd \eta_{m-2} \lhd \ldots \lhd \eta_1 \lhd \eta_0 = \mu \]
	be the normal closure series of $\eta$. Then, by lemma \ref{prn_subgp}, $\eta$ is a pronormal $L$-subgroup of $\eta_{m-2}$. Also, $\eta$ is a subnormal $L$-subgroup of $\eta_{m-2}$ with defect $2$. Therefore $\eta$ is normal in $\eta_{m-2}$. By the definition of normal closure, $\eta_{m-1}$ is the smallest normal $L$-subgroup of $\eta_{m-2}$ containing $\eta_m = \eta$. Since $\eta$ is a normal $L$-subgroup of $\eta_{m-2}$, we must have $\eta_{m-1} = \eta$. Thus
	\[ \eta = \eta_{m-1} \lhd \eta_{m-2} \lhd \ldots \lhd \eta_1 \lhd \eta_0 = \mu \]
	is the normal closure series for $\eta$. Hence $\eta$ is a subnormal $L$-subgroup of $\mu$ with defect $m-1$ and by the induction hypothesis, $\eta$ is normal in $\mu$.
\end{proof}

\noindent One of the significant applications of Theorem \ref{subprn} is in the case of nilpotent $L$-subgroups. In \cite{jahan_app}, the authors have studied the ascending chain of normalizers  and normal closure series of $L$-subgroups of nilpotent $L$-subgroups in detail. The results presented in \cite{jahan_app}, along with Theorem \ref{subprn}, can be utilized to show that in nilpotent $L$-groups, for $L$-subgroups having the same tip and tail as the parent $L$-group, the notions of normal and pronormal $L$-subgroups coincide (Theorem \ref{nil_norprn}).  

The notion of a nilpotent $L$-subgroup was developed by Ajmal and Jahan \cite{ajmal_nil}. For this, the definition of the commutator of two $L$-subgroups was modified, and this modified definition was used to develop the notion of the descending central chain of an $L$-subgroup. We recall these concepts below.

\begin{definition}(\cite{ajmal_nil})
	Let $\eta$, $\theta \in L^{\mu}$. The commutator of $\eta$ and $\theta$ is the $L$-subset $(\eta, \theta)$ of $\mu$ defined as follows:
	\[ (\eta, \theta)(x) = 
	\begin{cases}
	\begin{split}
		\vee \{ \eta(y) &\wedge \theta(z)\} \\& \text{if }x=[y,z] \text{ for some } y,z \in G, \\ 
		\text{inf } \eta &\wedge \text{inf } \theta \\& \text{if } x \neq [y,z] \text{ for any } y,z \in G.
	\end{split}	
	\end{cases} \]
\end{definition}

\noindent The commutator $L$-subgroup of $\eta$, $\theta \in L^{\mu}$, denoted by $[\eta, \theta]$, is defined to be the $L$-subgroup of $\mu$ generated by $(\eta, \theta)$.

\begin{definition}(\cite{ajmal_nil})
	\label{def_centchain}
	Let $\eta \in L(\mu)$. Take $Z_0(\eta) = \eta$ and for each $i \geq 0$, define $Z_{i+1}(\eta) = [Z_i(\eta), \eta]$. Then, the chain
	\[ \eta = Z_0(\eta) \supseteq Z_1(\eta) \supseteq \ldots \supseteq Z_i(\eta) \supseteq \ldots \]
	of $L$-subgroups of $\mu$ is called the descending central chain of $\eta$. 
\end{definition}

\begin{definition}(\cite{ajmal_nil})
	\label{def_nil}
	Let $\eta \in L(\mu)$ with tip $a_0$ and tail $t_0$ and $a_0 \neq t_0$. If the descending central chain
	\[ \eta = Z_0(\eta) \supseteq Z_1(\eta) \supseteq \ldots \supseteq Z_i(\eta) \supseteq \ldots \]
	terminates to the trivial $L$-subgroup $\eta_{t_0}^{a_0}$ in a finite number of steps, then $\eta$ is called a nilpotent $L$-subgroup of $\mu$. Moreover, $\eta$ is said to be nilpotent of class $c$ if $c$ is the smallest non-negative integer such that $Z_c(\eta) = \eta_{t_0}^{a_0}$. 
\end{definition}

\noindent Here, we define the successive normalizers of $\eta$ as follows:
\[	\eta_0 = \eta \qquad \text{ and } \qquad \eta_{i+1} = N(\eta_i) \qquad \text{ for all } i \geq 0.	\]
Then, by the definition of normalizer (see Definition \ref{defn_norm1}),  $\eta_{i+1}$ is the the largest $L$-subgroup of $\mu$ containing $\eta_i$ such that  $\eta_i \lhd \eta_{i+1}$. Consequently,
\begin{equation}
	\label{chain1}
	\eta = \eta_0 \subseteq \eta_1 \subseteq \ldots \subseteq \eta_i \subseteq \eta_{i+1} \subseteq \ldots 
\end{equation}
is an ascending chain of $L$-subgroups of $\mu$ starting from $\eta$ such that each $\eta_i$ is a normal $L$-subgroup of $\eta_{i+1}$. We call (\ref{chain1}) the ascending chain of normalizers of  $\eta$ in $\mu$.

\begin{lemma}(\cite{jahan_app})
	\label{subnormal}
	Let $\mu \in L(G)$  be a nilpotent $L$-group and $\eta$ be an $L$-subgroup of $\mu$ having the same tip and tail as $\mu$. Then, the ascending chain of normalizers of  $\eta$ in $\mu$ is finite and terminates at $\mu$.
\end{lemma}

\begin{lemma}(\cite{jahan_app})
	\label{chain_nc}
	Let $\mu \in L(G)$ and $\eta \in L(\mu)$. Then, there exists an ascending chain of $L$-subgroups 
	\[ \eta = \theta_0 \subseteq \theta_1 \subseteq \ldots \subseteq \theta_n \subseteq \ldots \subseteq  \mu \]
	terminating at $\mu $ in a  finite number of steps such that each $\theta_i$ is normal in $\theta_{i+1}$ if and only if the normal closure series of $\eta$ in $\mu$ terminates at $\eta$ in a  finite number of steps.
\end{lemma}

\noindent In view of Lemmas \ref{subnormal} and \ref{chain_nc}, we have the following:

\begin{theorem}\label{nil_subnormal}
	Let $\mu \in L(G)$ be a nilpotent $L$-group and $\eta$ be an $L$-subgroup of $\mu$ having the same tip and tail as $\mu$. Then, $\eta$ is a subnormal $L$-subgroup of $\mu$.
\end{theorem}

\noindent Thus Theorems \ref{nor_prn}, \ref{subprn} and \ref{nil_subnormal} together give the following result:

\begin{corollary}\label{nil_norprn}
	Let $\mu \in L(G)$ be a nilpotent $L$-group and $\eta$ be an $L$-subgroup of $\mu$ having the same tip and tail as $\mu$. Then, $\eta$ is a normal $L$-subgroup of $\mu$ if and only if $\eta$ is a pronormal $L$-subgroup of $\mu$.
\end{corollary}

\section{Conclusion}
In classical group theory, the pronormal subgroups play an indispensable role in the studies of normality and subnormality. While the notion of the normal $L$-subgroups and subnormal $L$-subgroups were efficiently introduced in \cite{wu_normal} and \cite{ajmal_subnormal}, respectively, the concept of pronormal $L$-subgroup that was compatible with these notions was absent. We have, in this paper, succeeded in providing such a notion of pronormality. Moreover, the notion of the pronormal $L$-subgroup developed in this study can be applied in the studies of the concepts of abnormal and contranormal $L$-subgroups, etc. These notions are closely related to the concept of normality. However, a proper research on these topics is lacking due to the absence of a  comprehensive recent study of the pronormal $L$-subgroups. The pronormality developed in this paper has removed this limitation and opens the door to research on these topics. 

The research in the discipline of fuzzy group theory came to a halt after Tom Head’s metatheorem and subdirect product theorems. This is because most of the concepts and results in the studies of fuzzy algebra could be established through simple applications of the metatheorem and the subdirect product theorem. However, the metatheorem and the subdirect product theorems are not applicable in the $L$-setting. Hence we suggest the researchers pursuing studies in these areas to investigate the properties of $L$-subalgebras of an $L$-algebra rather than $L$-subalgebras of classical algebra. 

\section*{Acknowledgements}
\noindent The second author of this paper was supported by the Senior Research Fellowship jointly funded by CSIR and UGC, India during the course of development of this paper.

\end{document}